\UseRawInputEncoding

\documentclass{amsart}

\usepackage{graphicx} 
\usepackage{geometry,graphicx,amssymb,amsmath,amsthm,amsbsy,eucal,amsfonts,mathrsfs,amscd,bm}
\usepackage{color}

\usepackage{tikz}
\usetikzlibrary{arrows,shapes}
\usetikzlibrary{arrows, decorations.markings,fit}
\usetikzlibrary{calc,3d}
\usepackage{lipsum}  
\usepackage{tikz-3dplot}

\usepackage{pgfplots}

\usepackage[all,cmtip]{xy}

\newtheorem{theorem}{Theorem}[section] 
\newtheorem{lemma}[theorem]{Lemma}

\newtheorem{corollary}[theorem]{Corollary}

\theoremstyle{definition}

\theoremstyle{remark}
\newtheorem{remark}[theorem]{Remark}

\numberwithin{equation}{section}

\allowdisplaybreaks[4]

\newcommand{\p}{{\partial}}  

\newcommand{\nab}{\nabla}

\newcommand{\mct}{\mathcal{T}_h}

\newcommand{\tbar}[1]{|\hspace{-0.03cm}|\hspace{-0.03cm}|#1|\hspace{-0.03cm}|\hspace{-0.03cm}|}

\newcommand{\bH}{{\bm H}}
\newcommand{\Div}{{\rm div}\,}

\newcommand{\bv}{{\bm v}}

\newcommand{\bV}{{\bm V}}

\newcommand{\bz}{\bm z}

\newcommand{\bbR}{\mathbb{R}}

\newcommand{\pol}{\EuScript{P}}
\newcommand{\bpol}{\boldsymbol{\pol}}

\newcommand{\calV}{\mathcal{V}}

\newcommand{\bw}{\bm w}
\newcommand{\bn}{\bm n}
\newcommand{\bt}{\bm t}
\newcommand{\bu}{\bm u}
\newcommand{\bZ}{\bm Z}

\newcommand{\calE}{\mathcal{E}}

\newcommand{\calM}{\mathcal{M}}
\newcommand{\calT}{\mathcal{T}}
\newcommand{\calS}{\mathcal{S}}
\newcommand{\mc}{\mathfrak{c}}
\newcommand{\bd}{{\bm d}}
\newcommand{\bp}{{\bm p}}
\newcommand{\by}{{\bm y}}

\date{}

\title[Divergence-free FEM with boundary correction]{A divergence-free finite element method for the Stokes problem with boundary correction}

\author[H.~Liu, M.~Neilan, and M.B.~Otus]{Haoran Liu\email{HAL104@pitt.edu} \and Michael Neilan\email{neilan@pitt.edu} \and M. Baris Otus \email{MBO13@pitt.edu}}
 
 \address{Department of Mathematics, University of Pittsburgh, Pittsburgh, PA 15260}

 \thanks{This work was supported in part by the National Science Foundation through
 grant number DMS-2011733.}

 

\begin{document}

\maketitle

\begin{abstract}
This paper constructs and analyzes a boundary correction finite element method for the Stokes problem 
based on  the Scott-Vogelius pair on Clough-Tocher splits. The velocity space consists of continuous piecewise quadratic polynomials, and the pressure space consists of piecewise linear polynomials without  continuity constraints. A Lagrange multiplier space that consists of continuous piecewise quadratic polynomials with respect to boundary partition is introduced to enforce boundary conditions as well as to mitigate the lack of pressure-robustness.  We prove several inf-sup conditions, leading
to the  well-posedness of the method.  In addition, we show that the method converges with optimal order
and the velocity approximation is divergence free.
\end{abstract}


\section{Introduction}

Boundary correction methods are a broad class of unfitted finite element methods, i.e.,
methods in which the computational mesh does not conform to the physical domain $\Omega$.
In contrast to, e.g., isoparametric methods, in which a domain is approximated
via curved elements, boundary correction methods generally solve a PDE
in a polytopal interior domain and transfer boundary conditions in a way such 
that the scheme still maintains optimal order convergence.
This polytopal approximation is, in general, not an $O(h^2)$ approximation
to the physical domain and in particular, the polytope's vertices are not necessarily
on the boundary of $\Omega$.
This approach can be advantageous for, e.g., dynamic problems with moving boundaries,
as remeshing is not needed at each time step.  
Another feature of boundary correction methods, in contrast to other unfitted schemes,  is the absence of `cut elements'
which may require special quadrature formula and algebraic stabilization.
Boundary correction methods were first
introduced and analyzed nearly $50$ years ago \cite{BDT72}
for the Poisson problem, and the technique has been
improved and refined recently resulting in practical and robust implementations 
\cite{CockburnSolano12,OSZ20,MainScovazzi18,ACS20A,ACS20B,ACS20C}
(see also \cite{BK94,DGS20} for variants).

In this article, we construct
a  boundary correction finite element method for the Stokes problem
based on the Scott-Vogelius pair on Clough-Tocher (or Alfeld) splits.  The velocity approximation is sought
in the space of continuous piecewise quadratic polynomials, whereas 
the pressure space is approximated by piecewise linear polynomials
without continuity constraints.  From their definitions, we see that the divergence operator maps 
the velocity space into the pressure space, and therefore, the scheme
yields divergence-free velocity approximations.  As far as we are aware
this is the first $H^1$-conforming divergence--free finite element method for incompressible flow
on unfitted meshes.

The construction and analysis of divergence-free methods
is an active area of research, and many
schemes have been proposed \cite{JohnetAl17,ArnoldQin92,Zhang05,NeilanFalk13,GuzmanNeilan14,GuzmanNeilan18}. 
 These schemes have several inherent advantages, e.g., exact conservation laws
for any mesh size and long-time stability \cite{RebholzEtAl17,BelenliRebTone15}.  Another feature of these schemes
is pressure-robustness; similar to the continuous setting, modifying the source term in the Stokes problem
by a gradient field only affects the pressure approximation.  This feature leads
to a decoupling in the velocity error, with abstract estimates independent of the viscosity.
Thus, divergence-free schemes may be advantageous for high Reynold number
flows and/or flows with large pressure gradients \cite{SchroederEtal18,SchroederLube18,Linke09}.   Except for the recent work \cite{NeilanOtus21}, where
isoparametric methods are introduced and studied, all of these divergence--free methods are applied to PDEs
on polytopal domains.

Let us describe the scheme in more detail and briefly summarize the context
of our results.  
The method starts with a background mesh enveloping the domain $\Omega$,
and the computational mesh simply consists of those
elements fully contained in $\bar\Omega$.  The method
is based on a standard Nitsche-based formulation,
where the Dirichlet
boundary conditions are enforced via penalization.
As the computational domain does not conform 
to $\Omega$, boundary conditions are corrected
via simple applications of Taylor's theorem to reduce
the inconsistency of the scheme.   

The procedure
described so far is relatively standard for the Poisson problem (cf.~\cite{BDT72,MainScovazzi18,ACS20A,ACS20B,ACS20C}),
but leads to some pressing issues for the Stokes equations.
First, because the computational domain explicitly depends on the mesh parameter $h$,
inf-sup stability of the Stokes pair is not immediately obvious.
As explained in \cite{GuzmanMaxim18},
the standard proof of inf-sup stability in the continuous setting (which
is needed for the discrete result) is based on a decomposition of the computational domain
into a finite number of strictly star shaped domains; the number
of star shaped domains is generally unbounded as $h\to 0$.
This issue can be circumvented with pressure-stabilization \cite{MainScovazzi18,ACS20B}, but at the price
of additional consistency errors and poor conservation properties.
We address this stability issue by designing the computational mesh
such that it inherits a macro element structure and 
applying the framework developed in \cite{GuzmanMaxim18}
for Stokes pairs on unfitted domains.  Doing so, we show
that the resulting pair is uniformly stable on the unfitted domain with respect
to the discretization parameter.

 The second difficulty of a boundary correction
method for the Stokes problem is its lack of pressure-robustness.
This feature is not due to the boundary correction per se, but rather
due to the weak enforcement of boundary conditions via penalization.
In particular, a divergence-free method for the Stokes problem
with weak enforcement of the boundary conditions is {\em not} pressure robust.
We mitigate the lack of pressure robustness in the scheme by introducing
an additional Lagrange multiplier that enforces the boundary conditions of the normal component of the velocity.
The Lagrange multiplier space consists of continuous piecewise quadratic polynomials
with respect to the boundary partition, and the Lagrange multiplier is an approximation to the pressure (modulo
an additive constant) restricted to the computational boundary.  
The Lagrange multiplier ameliorates the lack of pressure robustness of the method
and leads to a weakly coupled velocity error estimate; the velocity error's dependence on the viscosity
is compensated by a higher-order power of the discretization parameter $h$.
We remark that Lagrange multipliers within boundary correction schemes have been proposed and studied in \cite{BurmanHansboLarson20,CPBG19}
for the Poisson problem.

The rest of the paper is organized as follows.
In the next section, we state the Stokes problem, 
the computational mesh, and the boundary transfer operator.
In Section \ref{sec-FEM}, we state the finite element method
and show that the scheme yields exactly divergence--free velocity approximations.
Section \ref{sec-stability} proves several inf-sup conditions and 
the well-posedness of the method.
In Section \ref{sec-convergence}, we prove optimal order convergence provided
the exact solution is sufficiently smooth.  Finally, in Section \ref{sec-Numerics}
we perform some numerical experiments which verify the theoretical results,
and give some concluding remarks in Section \ref{sec-conclusion}.

\section{Preliminaries}
For a two-dimensional bounded domain $\Omega\subset \bbR^2$,
we consider the Stokes problem
\begin{subequations}
\label{eqn:Stokes}
\begin{alignat}{2}
\label{eqn:Stokes1}
-\nu \Delta \bu + \nab p & = {\bm f}\qquad &&\text{in }\Omega,\\
\Div \bu & = 0\qquad &&\text{in }\Omega,\\
\bu& = {\bm g}\qquad &&\text{on }\p \Omega,
\end{alignat}
\end{subequations}
where $\nu>0$ is the viscosity, assumed to be constant.
For simplicity in the presentation, and without loss of generality, we assume
that ${\bm g}=0$.  The extension to non-homogeneous boundary conditions
is relatively straight-forward \cite{HeisterLeo16}.

We assume the domain has smooth boundary $\p\Omega$ with outward unit normal
$\bn$.
We denote by $\phi$ the signed distance
function of $\Omega$ such that $\phi(x)<0$ for $x\in \Omega$
and $\phi(x)\ge 0$ otherwise, so that $\bn = \nab \phi/|\nab \phi|$
on $\p\Omega$.  For a positive number $\tau$,
denote by $\Gamma_\tau = \{x\in \bbR^2: |\phi(x)|\le \tau|\}$
the tubular region around $\p\Omega$.
By \cite[Lemma 14.16]{GilbargTrudinger},
there exists $\tau_0>0$ such
the closest point projection $\bp:\Gamma_{\tau_0}\to \p\Omega$
is well defined and satisfies
$\bp(x) = x - \phi(x)\bn(\bp(x))$ for all $x\in \Gamma_{\tau_0}$ \cite{BurmanHansboLarson20}.

Let $S\subset \bbR^2$ be a polygon
such that $\Omega\subset S$, and let $\calS_h$
be a quasi-uniform triangulation of $S$ that consists of shape regular triangles.  
We  define the computational mesh
as
\[
\mct = \{T\in \calS_h:\ \bar T\subset \bar \Omega\},
\]
and set
\[
\Omega_h = {\rm int}\Big(\bigcup_{T\in \mct} \bar T\Big)\subset \Omega
\]
to be the associated domain.
We denote by $\calT_h^{ct}$ the Clough-Tocher refinement of $\mct$, obtained
by connecting the vertices of each $T\in \mct$ to its barycenter.
The set of boundary of edges of $\calT_h$, which is also the set of boundary
edges of $\calT_h^{ct}$, is denoted by $\calE^B_h$.
With an abuse of notation, for a piecewise smooth function $q$ (with respect to $\calE_h^B$), we
write
\[
\int_{\p\Omega_h} q\, ds = \sum_{e\in \calE_h^B} \int_e q\, ds.
\]
We use $\bn_h$ to denote the outward unit normal with respect
to the computational boundary $\p\Omega_h$.
For $K\in \calT_h^{ct}$, we set $h_K = {\rm diam}(K)$
and $h = \max_{K\in \calT_h^{ct}} h_K$.
Likewise, for $e\in \calE_h^B$, we set $h_e = {\rm diam}(e)$.

\begin{remark}
Denote by $\calS_h^{ct}$ the Clough-Tocher refinement
of the background mesh $\calS_h$.
We emphasize that $\calT_h^{ct}\subset \calS_h^{ct}$, however,
\[
\calT_h^{ct}\neq \{K\in \calS_h^{ct}:\ \bar K\subset \bar\Omega\}.
\]
In particular, $\calT_h^{ct}$ inherits
the macro-element structure needed to prove
the stability of the Scott-Vogelius pair.
\end{remark}

\subsection{Boundary transfer operator}
The main component of boundary correction methods
is a well-defined mapping $M:\p\Omega_h\to \p\Omega$ that assigns
each point on the computational boundary to physical one in order to ``transfer'' the boundary
information on $\p\Omega$ to $\p\Omega_h$.
With such a mapping in hand, we can define
the transfer direction  as
\[
\frak{\bd}({x})=(M-I){x}\qquad x\in \p\Omega_h,
\]
and transfer length
\begin{equation}\label{eqn:deltaDef}
\delta(x)  = |\frak{\bd}(x)|.
\end{equation}

Several choices of the mapping $M$
and corresponding transfer directions have appeared in 
the literature. A common choice
(and arguably the most natural)
is to take $M$ to be the closest point projection, i.e., $M = \bp$.
In this case, assuming $\Omega_h$ approximates 
$\Omega$ well enough, the distance vector $\frak{\bd}$ defined above coincides (up to a multiplicative constant)
with the outward unit normal vector $\bn$ of the original boundary $\partial \Omega$.
In particular, there holds $\frak{\bd}(x) = \phi(x)\bn(\bp(x))$ and $\delta(x) = |\phi(x)|$.
Another common choice is to take the transfer direction to be parallel to the outward unit normal
 of the computational boundary, i.e., $\frak{\bd}/\delta =   \bn_h$.
 In this case, we have $\delta(x)\ge |\phi(x)|$ with possible large discrepancies between $\delta(x)$ and $|\phi(x)|$, but it leads
 to a simpler implementation in the numerical method.

In the definition and analysis of the method
below, we do not explicitly define the mapping $M$;
rather, our main requirement
for the mapping $M$ is to satisfy the Assumption \eqref{eqn:Assumption} below.
In particular, and similar to \cite{BDT72,OSZ20,BurmanHansboLarson20,ACS20A,ACS20B,ACS20C}, the stability and convergence analysis
only assumes that the transfer distance $\delta(x)$ is sufficiently
small relative to the mesh parameter $h$.
In the numerical experiments provided in Section \ref{sec-Numerics},
we take $M$ to be an approximation to the closest point projection.

Set $\bd = \frak{\bd}/{\delta}$, 
for $x \in \partial \Omega_h$, and define the boundary transfer operator
\[
(S_h \bv)(x) = \bv(x)+\delta(x) \frac{\p \bv}{\p \bd}(x)+ \frac12 (\delta(x))^2 \frac{\p^2 \bv}{\p \bd^2}(x).
\]
Note that  $(S_h\bv)(x)$ is the second-order Taylor expansion of the function $\bv$.
\begin{remark}
Throughout this paper, the  constants $C$ and $c$ (with or without subscripts) denote some positive constants that are independent of the mesh parameter $h$ and the viscosity.
\end{remark}

\section{A divergence--free finite element method}\label{sec-FEM}
For $D\subset \bbR^d$,
denote by $\pol_k(D)$ the space of polynomials
of degree $\le k$ with domain $D$.
Analogous vector-valued spaces are denoted in boldface.
We define the lowest-order Scott-Vogelius finite element pair
with respect to the Clough-Tocher triangulation $\mct^{ct}$:
\begin{align*}
\bV_h & = \{\bv\in \bH^1(\Omega_h):\ \bv|_K\in \bpol_2(K)\ \forall K\in \mct^{ct},\ \int_{\p\Omega_h} (\bv\cdot \bn_h)\, ds=0\},\\
Q_h & = \{q\in L^2(\Omega_h):\ q|_K\in \pol_1(K)\ \forall K\in \mct^{ct}\},
\end{align*}
and the analogous spaces with boundary conditions
\begin{align*}
\mathring{\bV}_h & = \bV_h\cap \bH^1_0(\Omega_h),\qquad
\mathring{Q}_h  = Q_h\cap L^2_0(\Omega_h).
\end{align*}
We further introduce a Lagrange multiplier space 
\begin{align*}
X_h = \{\mu\in C(\p\Omega_h):\ \mu|_e \in \pol_2(e)\ \forall e\in \calE_h^B\},
\end{align*}
and its variant,
\[
\mathring{X}_h = \{\mu\in X_h:\ \int_{\p\Omega_h} \mu\, ds = 0\}.
\]

We define the bilinear form
\begin{align*}
a_h(\bu,\bv) &=\nu\Big( \int_{\Omega_h} \nab \bu :\nab \bv\, dx -  \int_{\p\Omega_h} \frac{\p \bu}{\p \bn_h} \cdot \bv\, ds
+ \int_{\p\Omega_h} \frac{\p \bv}{\p \bn_h}\cdot (S_h \bu)\, ds,\\
&\qquad + \sum_{e\in \calE_h^B} \int_{e} \frac{\sigma}{h_e} (S_h \bu)\cdot (S_h \bv)\, ds\Big),
\end{align*}
 where $\sigma>0$ is a penalty parameter.

\begin{remark}
The bilinear form $a_h(\cdot,\cdot)$ is based on a standard ``Nitsche bilinear form'' associated
with the Laplace operator, but with boundary correction \cite{Nitsche71,RiviereBook}.
Note that the bilinear form 
is based on a non-symmetric version of Nitsche's method
due to the positive sign in front of the third term
in the bilinear form $a_h(\cdot,\cdot)$.
However, boundary correction methods based on the symmetric version
of Nitsche's method still yield a non-symmetric
bilinear form \cite{BDT72,MainScovazzi18}.
The non-symmetric version allows less restrictions on the penalty parameter $\sigma$ to ensure
stability if the extension direction coincides with the outward unit
normal of $\p\Omega_h$.
In particular, if $\bd = \bn_h$,
a standard argument shows that the bilinear form $a_h(\cdot,\cdot)$ is coercive on $\bV_h$
for any $\sigma>0$; cf.~Lemma \ref{lem:coercivity}.
\end{remark}

We define two bilinear forms associated with the continuity equations, one without and one with 
boundary correction:
\begin{align*}
  b_h(\bv,(q,\mu)) &= -\int_{\Omega_h} (\Div \bv)q\, dx + \int_{\p\Omega_h} (\bv\cdot \bn_h)\mu\, ds,\\  
  b^e_h(\bv,(q,\mu)) &= -\int_{\Omega_h} (\Div \bv)q\, dx +  \int_{\p\Omega_h} ((S_h\bv)\cdot \bn_h)\mu\, ds.
 \end{align*}

We consider the method of finding
$(\bu_h,p_h,\lambda_h)\in \bV_h\times \mathring{Q}_h\times \mathring{X}_h$ such that
\begin{subequations}
\label{eqn:LMM}
\begin{alignat}{2}
\label{eqn:LMM1}
a_h(\bu_h,\bv) +b_h(\bv,(p_h,\lambda_h)) & = \int_{\Omega_h} {\bm f}\cdot \bv\, dx
\qquad &&\forall \bv\in \bV_h,\\
\label{eqn:LMM2}
b^e_h(\bu_h,(q,\mu)) & = 0
\qquad &&\forall (q,\mu)\in \mathring{Q}_h\times \mathring{X}_h.
\end{alignat}
\end{subequations}

\begin{remark}
The zero mean-value constant
defined in the Lagrange multiplier space $\mathring{X}_h$
mods out constants, and is due to the condition $\int_{\p\Omega_h} (\bv\cdot \bn_h)\,ds=0$
in the definition of the discrete velocity space $\bV_h$.  
If this constraint is not imposed in the Lagrange multiplier space,
then in general \eqref{eqn:LMM} is ill-posed since
\[
b_h(\bv,(0,1)) = 0\qquad \forall \bv\in \bV_h.
\]
On the other hand, the constraint $\int_{\p\Omega_h} (\bv\cdot \bn_h)\, ds=0$
is needed to ensure that method \eqref{eqn:LMM} yields
a divergence-free solution, as the next lemma shows.
\end{remark}

\begin{lemma}[Divergence--free property]\label{lem:DivFree}
If $(\bu_h,p_h,\lambda_h)\in \bV_h\times \mathring{Q}_h\times \mathring{X}_h$
satisfies \eqref{eqn:LMM}, then $\Div \bu_h\equiv 0$ in $\Omega_h$.
\end{lemma}
\begin{proof}
The definition of the Stokes pair $\bV_h\times \mathring{Q}_h$
shows $\Div \bu_h\in \mathring{Q}_h$.
Then, letting $q=\Div \bu_h$ and  $\mu=0$ in \eqref{eqn:LMM2} yields
\[
0 = b^e_h(\bu_h,(\Div \bu_h,0))  = -\|\Div \bu_h\|_{L^2(\Omega_h)}^2.
\]
Thus, $\Div \bu_h \equiv 0$.
\end{proof}

\section{Stability and Continuity estimates}\label{sec-stability}


In our stability and convergence analysis,
we make an assumption regarding the distance
between the PDE domain $\Omega$
and the computational domain $\Omega_h$.
To state this assumption, we define for a boundary edge $e\in \calE_h^B$,
\[
\delta_e := \max_{x\in \bar e} \delta(x).
\]
We make the assumption
\begin{equation}\tag{A}\label{eqn:Assumption}
\max_{e\in \calE_h^B} h_e^{-1} \delta_e \le c_\delta<1,\qquad \text{for }c_\delta\text{ sufficiently small.}
\end{equation}
\begin{remark}
 Assumption \eqref{eqn:Assumption} essentially states
 that the distance between $\p\Omega$ and $\p\Omega_h$ is of order $h$, i.e.,
$\delta = O(h)$ with (hidden) constant
sufficiently small.  Similar assumptions, in the context
of boundary correction methods, are made in, e.g.,
\cite{BDT72,OSZ19,MainScovazzi18,ACS20A,ACS20B}.
\end{remark}

We define three $H^1$-type norms on $\bV_h$:
\begin{align*}
\|\bv\|_{h}^2 
& = \|\nab \bv\|_{L^2(\Omega_h)}^2 +\sum_{e\in \calE_h^B} h_e^{-1}\|S_h \bv\|_{L^2(e)}^2,\\
\|\bv\|_{1,h}^2
& = \|\nab \bv\|_{L^2(\Omega_h)}^2 + \sum_{e\in \calE_h^B} h_e^{-1} \|\bv\|_{L^2(e)}^2,\\
\tbar{\bv}^2_{h}
&= \|\bv\|_{h}^2 + \sum_{e\in \calE_h^B} h_e \|\nab \bv\|_{L^2(e)}^2.
\end{align*}
In addition, we define a $H^{-1/2}$-norm on the Lagrange multiplier space $\mathring{X}_h$:
\[
\|\mu\|_{-1/2,h}^2  = \sum_{e\in \calE_h^B} h_e \|\mu \|_{L^2(e)}^2.
\]
Finally, we define the norm on $\mathring{Q}_h\times \mathring{X}_h$ as
\[
\|(q,\mu)\|:=\|q\|_{L^2(\Omega_h)}+\|\mu\|_{-1/2,h}.
\]

\begin{lemma}\label{lem:EquivNorms}
There holds for all $\bv\in \bV_h$,
\begin{align}\label{eqn:ShvDiff}
\sum_{e\in \calE_h^B} h_e^{-1} \|S_h \bv-\bv\|_{L^2(e)}^2\le C c_\delta \|\nab \bv\|_{L^2(\Omega_h)}^2,\\
\nonumber \sum_{e\in \calE_h^B} h_e^{-1} \|S_h \bv\|_{L^2(e)}^2\le C \|\bv\|_{1,h}^2,
\end{align}
provided that $c_\delta$ in \eqref{eqn:Assumption} is sufficiently small.
In particular, $\|\cdot\|_h$, $\|\cdot\|_{1,h}$, and $\tbar{\cdot}_{h}$ are equivalent on $\bV_h$.
\end{lemma}
\begin{proof}
By trace and inverse inequalities, the shape-regularity of $\mct$ and \eqref{eqn:Assumption}, there holds for $e\in \calE_h^B$,
\begin{align}\label{eqn:IT}
h_e^{-1} \int_e |\delta|^{2j} \big|\frac{\p^j \bv}{\p \bd^j}\big|^2\, ds\le C \delta_e^{2j} h_e^{-2j}\|\nab \bv\|_{L^2(T_e)}^2
\le C c_\delta^{2j}  \|\nab \bv\|_{L^2(T_e)}^2 \quad j=1,2,
\end{align}
where $T_e\in \calT_h$ satisfies $e\subset \p T$. It then follows from the definition of $S_h$ and
$\|\cdot\|_{1,h}$ that
\begin{align*}
\sum_{e\in \calE_h^B} h_e^{-1} \|S_h \bv\|_{L^2(e)}^2
&\le C \sum_{e\in \calE_h^B} \sum_{j=0}^2 h_e^{-1} \int_e |\delta|^{2j} \big|\frac{\p^j \bv}{\p \bd^j}\big|^2\, ds\le C \|\bv\|_{1,h}^2.
\end{align*}
This inequality immediately yields $\|\bv\|_h\le C \|\bv\|_{1,h}$.
Moreover, standard arguments involving the trace and inverse inequalities
show $\|\bv\|_h\le \tbar{\bv}_h\le C\|\bv\|_h$ on $\bV_h$.
Thus, to complete the proof,
it suffices to show that $\|\bv\|_{1,h}\le C\|\bv\|_h$.

To this end, we once again use \eqref{eqn:IT} to obtain
\begin{align*}
\sum_{e\in \calE_h^B} h_e^{-1} \|\bv\|_{L^2(e)}^2
&\le 2\sum_{e\in \calE_h^B} h_e^{-1} \|S_h \bv\|_{L^2(e)}^2
+ 2\sum_{e\in \calE_h^B} h_e^{-1} \|S_h \bv-\bv\|_{L^2(e)}^2\\
&\le 2\sum_{e\in \calE_h^B} h_e^{-1} \|S_h \bv\|_{L^2(e)}^2
+ C\sum_{e\in \calE_h^B} h_e^{-1} \sum_{j=1}^2\int_e |\delta|^{2j} \big|\frac{\p^j \bv}{\p \bd^j}\big|^2\, ds\\
&\le  2\sum_{e\in \calE_h^B} h_e^{-1} \|S_h \bv\|_{L^2(e)}^2+C \|\nab \bv\|_{L^2(\Omega_h)}^2.
\end{align*}
This inequality implies $\|\bv\|_{1,h}\le C\|\bv\|_h$.

\end{proof}

\subsection{Continuity and coercivity estimates of bilinear forms}
\begin{lemma}\label{lem:bhPert}
There holds
\begin{alignat}{2}
\label{eqn:continuity}
|a_h(\bv,\bw)|&\le c_2(1+\sigma) \nu \tbar{\bv}_{h} \tbar{\bw}_{h}\qquad &&\forall \bv,\bw\in \bV_h+H^3(\Omega_h),\\
\label{eqn:bCont}
\big|b_h(\bv,(q,\mu))\big| &\le C \|\bv\|_{1,h} \|(q,\mu)\|\quad &&\forall (q,\mu)\in \mathring{Q}_h\times \mathring{X}_h,\\
\label{eqn:bhPert}
\big| b_h(\bv,(q,\mu)) - b_h^e(\bv,(q,\mu))\big|
&\le C c_\delta \|\bv\|_{1,h} \|(q,\mu)\|\qquad &&\forall \bv\in \bV_h,\ \forall (q,\mu)\in \mathring{Q}_h\times \mathring{X}_h.
\end{alignat}
\begin{proof}
The proof of the continuity estimate of \eqref{eqn:continuity}
is given in \cite[Proposition 1]{ACS20A} (with superficial modifications).
The continuity estimate  of $b_h(\cdot,\cdot)$ \eqref{eqn:bCont} follows directly
from the Cauchy-Schwarz inequality.

This third estimate \eqref{eqn:bhPert}  follows from the definition of the forms,
the Cauchy-Schwarz inequality, and \eqref{eqn:IT}:
\begin{align*}
\big|b_h(\bv,(q,\mu)) - b_h^e(\bv,(q,\mu))\big|
& = \Big|\sum_{e\in \calE_h^B} \int_e \big((\bv-S_h \bv)\cdot \bn_h\big) \mu\, ds\Big|\\
&\le C\Big(\sum_{e\in \calE_h^B} \sum_{j=1}^2  h_e^{-1} \int_e |\delta|^{2j} \big|\frac{\p^j \bv}{\p \bd^j}\big|^2\Big)^{1/2} \|\mu\|_{-1/2,h}\\
&\le C c_\delta \|\bv\|_{1,h} \|\mu\|_{-1/2,h}.
\end{align*}

\end{proof}
\end{lemma}

\begin{lemma}\label{lem:coercivity}
Suppose that Assumption \eqref{eqn:Assumption} is satisfied for $c_\delta$
sufficiently small.
Then there exists $\sigma_0>0$ such that, for $\sigma\ge \sigma_0$,
\[
c_1{\nu} \|\bv\|^2_{1,h} \le a_h(\bv,\bv)\qquad \forall \bv\in \bV_h
\]
for $c_1>0$ independent of $h$ and $\nu$.
If the extension direction $\bd$ coincides with the outward unit normal
of $\p\Omega_h$, i.e., if $\bd = \bn_h$, then coercivity is satisfied
for any positive penalty parameter $\sigma>0$.
\end{lemma}
\begin{proof}
The proof the result under
assumption \eqref{eqn:Assumption}
follows exactly from the arguments in \cite[Theorem 2]{ACS20A} (see also \cite[Lemma 6]{BDT72}),
so the proof is omitted.

If $\bd = \bn_h$, then by definition of the bilinear form $a_h(\cdot,\cdot)$,
\begin{align*}
a_h(\bv,\bv) 
&=\nu\Big( \|\nab \bv\|_{L^2(\Omega_h)}^2 
+\sum_{e\in \calE_h^B}\Big( \int_{e} \frac{\p \bv}{\p \bn_h}\cdot (S_h \bv-\bv)\, ds
+ \frac{\sigma}{h_e} \|S_h \bv\|_{L^2(e)}^2\Big)\Big)\\
& = \nu\Big( \|\nab \bv\|_{L^2(\Omega_h)}^2 
+\sum_{e\in \calE_h^B} \Big(\int_{e} \delta \big|\frac{\p \bv}{\p \bn_h}\big|^2\, ds
+ \frac12 \int_{e} \delta^2 \frac{\p \bv}{\p \bn_h} \frac{\p^2 \bv}{\p \bn^2_h}\, ds
+ \frac{\sigma}{h_e} \|S_h \bv\|_{L^2(e)}^2\Big)\Big).
\end{align*}
We then use the Cauchy-Schwarz inequality, standard trace and inverse estimates,
and Assumption \eqref{eqn:Assumption} to get
\begin{align*}
\sum_{e\in \calE_h^B} \int_{e} \delta^2 \frac{\p \bv}{\p \bn_h} \frac{\p^2 \bv}{\p \bn^2_h}\, ds
&\le (\max_{e\in \calE_h^B} h_e^{-1} \delta_e\big)^2\Big(\sum_{e\in \calE_h^B} h_e \|\nab \bv\|_{L^2(e)}^2\Big)^{1/2}  
\Big(\sum_{e\in \calE_h^B} h^{3}_e \|D^2 \bv\|_{L^2(e)}^2\Big)^{1/2}\\
&\le C c_\delta^2 \|\nab \bv\|_{L^2(\Omega_h)}^2.
\end{align*}

Thus, we find
\begin{align*}
a_h(\bv,\bv) 
& \ge \nu\Big( (1-C c_\delta^2)\|\nab \bv\|_{L^2(\Omega_h)}^2 
+ \frac{\sigma}{h_e} \|S_h \bv\|_{L^2(e)}^2\Big)\Big) \ge C \nu \|\bv\|^2_h\ge C \nu \|\bv\|_{1,h}^2
\end{align*}
for $c_\delta$ sufficiently small and for $\sigma>0$.

\end{proof}

\subsection{Inf-Sup Stability I}

In this section we prove
the discrete inf-sup (LBB) condition for the Stokes pair
$\mathring{\bV}_h\times \mathring{Q}_h$ 
with stability constants independent of $h$.
In the case of a fixed polygonal domain,
the LBB stability for this pair is well-known (cf.~\cite{ArnoldQin92,QinThesis94,GuzmanNeilan18});
however, the extension of these results to the unfitted domain $\Omega_h$ is not immediate.
In particular, the proofs in \cite{ArnoldQin92,QinThesis94,GuzmanNeilan18} (directly or indirectly)
rely on the Ne\v{c}as inequality:
\[
\mc_h \|q\|_{L^2(\Omega_h)}\le \sup_{\bv\in \bH^1_0(\Omega_h)\backslash \{0\}} \frac{\int_{\Omega_h} (\Div \bv)q\, dx}{\|\nab \bv\|_{L^2(\Omega_h)}}\qquad \forall q\in L^2_0(\Omega_h)
\]
for some $\mc_h>0$ depending on the domain $\Omega_h$.
As explained in  \cite{GuzmanMaxim18}, it is unclear if the constant $\mc_h$ in this inequality
is independent of $h$.

Our approach is to simply combine the local 
stability of the Scott-Vogelius 
pair with the stability of the $\bpol_2\times \pol_0$ pair.
For a (macro) element $T\in \calT_h$, we define the local
spaces with boundary conditions
\begin{align*}
\bV_0(T) &= \{\bv\in \bH^1_0(T):\ \bv|_K\in \bpol_2(K)\ \forall K\subset T,\ K\in \calT_h^{ct}\},\\
Q_0(T) & = \{q\in L^2_0(T):\ q|_K\in \pol_1(K)\ \forall K\subset T,\ K\in \calT_h^{ct}\}.
\end{align*}
We state a local surjectivity
of the divergence operator acting on these
spaces.  The proof is found in, e.g., \cite{GuzmanNeilan18}.
\begin{lemma}\label{lem:LocalInfSup}
For every $q\in Q_0(T)$, there exists 
$\bv\in \bV_0(T)$ such that $\Div \bv = q$
and $\|\nab \bv\|_{L^2(T)}\le \beta^{-1}_T \|q\|_{L^2(T)}$.
Here, the constant $\beta_T>0$ depends only
on the shape-regularity of $T$.
\end{lemma}

Next, we state the recent stability result of the $\bpol_2\times \pol_0$
pair on unfitted domains (cf.~\cite[Theorem 1, Section 6.3, and Remark 1]{GuzmanMaxim18}).
\begin{lemma}\label{lem:P2P0Stability}
Define the space of piecewise constants
with respect to the mesh $\calT_h$:
\[
\mathring{Y}_h = \{q\in L^2_0(\Omega_h):\ q|_T\in \pol_0(T)\ \forall T\in \calT_h\}\subset \mathring{Q}_h.
\]
There exists $\beta_0>0$ and $h_0>0$
such that for $h\le h_0$, there holds
\[
\sup_{\bv\in \mathring{\bV}_h\backslash \{0\}} \frac{\int_{\Omega_h} (\Div \bv)q\, dx}{\|\nab \bv\|_{L^2(\Omega_h)}} \ge \beta_0 \|q\|_{L^2(\Omega_h)}\qquad
\forall q\in \mathring{Y}_h.
\]
\end{lemma}

Combining Lemmas \ref{lem:LocalInfSup}--\ref{lem:P2P0Stability}
yields the following stability result for the $\mathring{\bV}_h\times \mathring{Q}_h$
Stokes pair.
\begin{lemma}\label{lem:LBB}
There exists $\beta_1>0$ independent of $h$ such that
\[
\sup_{\bv\in \mathring{\bV}_h\backslash \{0\}} \frac{\int_{\Omega_h} (\Div \bv)q\, dx}{\|\nab \bv\|_{L^2(\Omega_h)}} \ge \beta_1 \|q\|_{L^2(\Omega_h)}\qquad
\forall q\in \mathring{Q}_h.
\]
for $h\le h_0$.
\end{lemma}
\begin{proof}
The proof essentially follows from Lemmas \ref{lem:LocalInfSup}--\ref{lem:P2P0Stability}
with the arguments in \cite{ArnoldQin92,QinThesis94,GuzmanNeilan18}.  We provide
the proof for completeness.

Let $q\in \mathring{Q}_h$, and let $\bar q\in \mathring{Y}_h$ be its piecewise average, i.e., $\bar q|_T = |T|^{-1}\int_T q\, dx$ for all $T\in \calT_h$.
We then have $(q-\bar q)|_T\in Q_0(T)$ for all $T\in \calT_h$,
and therefore, by Lemma \ref{lem:LocalInfSup}, there exists $\bv_{1,T}\in \bV_0(T)$
such that $\Div \bv_{1,T} = (q-\bar q)|_T$ and $\|\nab \bv\|_{L^2(T)}\le \beta_T^{-1} \|q\|_{L^2(T)}$.
Defining $\bv_1\in \mathring{\bV}_h$ by $\bv_1|_T = \bv_{1,T}\ \forall T\in \calT_h$,
we have $\Div \bv_1 = (q-\bar q)$ in $\Omega_h$ and $\|\nab \bv_1\|_{L^2(\Omega_h)}\le \beta_*^{-1}\|q-\bar q\|_{L^2(\Omega_h)}$,
where $\beta_* = \min_{T\in \calT_h} \beta_T$.

With this result, and by Lemma \ref{lem:P2P0Stability}, we conclude
\begin{align*}
\beta_0 \|\bar q\|_{L^2(\Omega_h)}
& \le \sup_{\bv\in \mathring{\bV}_h\backslash \{0\}} \frac{\int_{\Omega_h} (\Div \bv)\bar q\, dx}{\|\nab \bv\|_{L^2(\Omega_h)}}\\
& \le \sup_{\bv\in \mathring{\bV}_h\backslash \{0\}} \frac{\int_{\Omega_h} (\Div \bv) q\, dx}{\|\nab \bv\|_{L^2(\Omega_h)}} +\|q-\bar q\|_{L^2(\Omega_h)}\\
%
%
& \le (1+\beta_*^{-1}) \sup_{\bv\in \mathring{\bV}_h\backslash \{0\}} \frac{\int_{\Omega_h} (\Div \bv) q\, dx}{\|\nab \bv\|_{L^2(\Omega_h)}}.
\end{align*}
Thus,
\begin{align*}
\|q\|_{L^2(\Omega_h)}\le \|q-\bar q\|_{L^2(\Omega_h)}+\|\bar q\|_{L^2(\Omega_h)}\le 
\big(\beta^{-1}_*+\beta_0^{-1}(1+\beta^{-1}_*)\big)\sup_{\bv\in \mathring{\bV}_h\backslash \{0\}} \frac{\int_{\Omega_h} (\Div \bv) q\, dx}{\|\nab \bv\|_{L^2(\Omega_h)}}.
\end{align*}
Setting $\beta_1 = \big(\beta^{-1}_*+\beta_0^{-1}(1+\beta^{-1}_*)\big)^{-1}$ completes the proof.
\end{proof}

\subsection{Inf-Sup Stability II}
The following lemma proves
inf-sup stability for the Lagrange multiplier part
of the bilinear form $b_h(\cdot,\cdot)$.
\begin{lemma}\label{lem:LMStab}
There holds 
\begin{align}\label{eqn:LinearInfSup}
\sup_{\bv\in \bV_h\backslash \{0\}}\frac{\int_{\p\Omega_h} (\bv\cdot \bn)\mu\, ds}{\|\bv\|_{1,h}} \ge \beta_2 \|\mu\|_{-1/2,h}\qquad\forall \mu\in \mathring{X}_h. 
\end{align}
for some $\beta_2>0$ independent of $h$.
\end{lemma}

\begin{proof}
We label the boundary edges
as $\{e_j\}_{j=1}^N = \calE_h^B$,
and denote the boundary vertices by $\{a_j\}_{j=1}^N = \calV_h^B$, labeled
such that $e_j$ has vertices $a_j$ and $a_{j+1}$, with the convention that $a_{N+1}=a_1$.
Define the set of boundary edge midpoints $\calM_h^B = \{m_j\}_{j=1}^N$
with $m_j = \frac12(a_j+a_{j+1})$.
Let $\bn_j$ be the normal vector of $\p\Omega_h$ restricted to the edge $e_j$,
and let $\bt_j$ be the tangent vector obtained by rotating $\bn_j$
$90$ degrees clockwise.
Without loss of generality, we assume that $\bt|_{e_j}$
is parallel to $a_{j+1}-a_j$. 
We further denote by $\calV_h^C$ the set of boundary corner vertices, i.e.,
if $a_j\in \calV_h^C$, then the outward unit normals $\bn_j,\bn_{j-1}$
of the edges touching $a_j$ are linearly independent.
The set of flat boundary vertices are defined as $\calV_h^F = \calV_h^B\backslash \calV_h^C$.
Note that $\bn_j = \bn_{j-1}$ and $\bt_j = \bt_{j-1}$ for $a_j\in \calV_h^F$.

Given $\mu\in \mathring{X}_h$, we define $\bv\in \bV_h$ by the conditions
\begin{equation}\label{eqn:vnDOFs}
\begin{aligned}
&(\bv\cdot \bn_j)(a_j) = h \mu(a_j),\quad &&(\bv\cdot \bn_{j-1})(a_j) = h \mu(a_j)\qquad &&\forall a_j\in \calV_h^C,\\
&(\bv\cdot \bn_j)(a_j) = h\mu(a_j),\quad &&(\bv\cdot \bt_j)(a_j) =0\qquad &&\forall a_j\in \calV_h^F,\\
&(\bv\cdot \bn_{j})(m_{j}) = h\mu(a_{j}),\quad &&(\bv\cdot \bt_j)(m_j)=0\qquad &&\forall m_j\in \calM_h^B.
\end{aligned}
\end{equation}
All other (quadratic Lagrange) degrees of freedom of $\bv$ are set to zero, i.e.,
$\bv(a)=0$ at all interior vertices and interior edge midpoints in $\calT_h^{ct}$.

Since $(\bv\cdot \bn_{j} - h \mu)|_{e_j}$ is a quadratic 
on each $e_j\in \calE_h^B$, and $\bv\cdot \bn_{j}=h\mu$ at three distinct points on $e_j$, we have that $\bv\cdot \bn_{j}-h\mu|_{e_j}=0$.
{Moreover, using quasi uniformity, we have} 
\begin{align}\label{ineq0}
\int_{\p \Omega_h}(\bv\cdot \bn)\mu\, ds\ge C\|\mu\|^2_{-1/2,h}.
\end{align}
It remains to show that $\|\bv\|_{1,h}\le C \|\mu\|_{-1/2,h}$ to complete the proof.

For $K\in \calT_h^{ct}$, let $\calV_K^B, \calV_K^C,\calV_K^F,\calM_K^B$
be the sets of elements in $\calV_h^B, \calV_h^C,\calV_h^F,\calM_h^B$
contained in $\bar K$, respectively.
By a standard scaling argument and \eqref{eqn:vnDOFs}, we get ($m=0,1$) 
\begin{align}\label{eqn:bvStart1}
\|\bv\|_{H^m(K)}^2
&\le C \sum_{a_j\in \calV_K^B\cup \calM_K^B} h_{e_j}^{2-2m} |\bv(a_j)|^2\\
%
&\nonumber= C \Big(\sum_{a_j\in \calV_K^C} h_{e_j}^{2-2m} |\bv(a_j)|^2
+ \sum_{a_j\in \calV_K^F\cup \calM_K^B} h_{e_j}^{4-2m} |\mu(a_j)|^2\Big).
\end{align}

Claim: $|\bv(a_j)|\le Ch|\mu(a_j)|$ for all $a_j\in \calV_K^C$, where $C>0$ is 
uniformly bounded and independent of $h$, $\bn_j$ and $\bn_{j-1}$.

{\em Proof of the claim}:  Assume that $\calV_K^C$ is non-empty
for otherwise the proof is trivial.
For $a_j\in \calV_K^C$, we write $\bv(a_j)$ in terms
of the basis $\{\bt_j,\bt_{j-1}\}$, use \eqref{eqn:vnDOFs},
and apply some elementary vector identities:
\begin{align*}
\bv(a_j) 
&= \frac1{\bt_{j-1}\cdot \bn_j} (\bv\cdot \bn_j)(a_j)\bt_{j-1}
+ \frac1{\bt_j\cdot \bn_{j-1}} (\bv\cdot \bn_{j-1})(a_j)\bt_j\\
&= h \mu(a_j)\Big(\frac1{\bt_{j-1}\cdot \bn_j}\bt_{j-1}
+ \frac1{\bt_j\cdot \bn_{j-1}} \bt_j\Big)\\
&= h \mu(a_j)\Big(
 \frac{\bt_j- \bt_{j-1}}{\bt_j\cdot \bn_{j-1}} \Big).
\end{align*}
Write $\bt_j = (\cos(\theta_j),\sin(\theta_j))^\intercal$.
We then compute $\bt_j \cdot \bn_{j-1} = \sin(\theta_{j-1} -\theta_j)$, and therefore
\begin{align*}
 \frac{\bt_j- \bt_{j-1}}{\bt_j\cdot \bn_{j-1}} 
 & = \frac{(\cos(\theta_j)-\cos(\theta_{j-1}),\sin(\theta_j)-\sin(\theta_{j-1}))^\intercal}{\sin(\theta_{j-1}-\theta_j)}.
 \end{align*}
Since
\begin{align*}
\lim_{\theta_j\rightarrow\theta_{j-1}}\frac{(\cos{\theta_j}-\cos{\theta_{j-1}}, \sin{\theta_j}-\sin{\theta_{j-1}})^\intercal}{\sin{(\theta_{j-1}-\theta_j)}}
&=\lim_{\theta_j\rightarrow\theta_{j-1}}\frac{(-\sin{\theta_j}, \cos{\theta_j})^\intercal}{-\cos{(\theta_{j-1}-\theta_j)}}
&=(\sin{\theta_{j-1}},\cos{\theta_{j-1}})^\intercal,
\end{align*}
we conclude that $\big| \frac{\bt_j- \bt_{j-1}}{\bt_j\cdot \bn_{j-1}} \big|$ is bounded
for $|\bt_j\cdot \bn_{j-1}|\ll 1$, i.e., for ``nearly flat boundary vertices''.
This conclusion and the shape regularity of the mesh
shows $\big| \frac{\bt_j- \bt_{j-1}}{\bt_j\cdot \bn_{j-1}} \big|\le C$
for some $C>0$ independent of $h$ and $\{\bn_{j-1},\bn_j\}$.
This concludes the proof of the claim.

Applying the claim to \eqref{eqn:bvStart1} and a scaling argument yields
\begin{align*}
\|\bv\|_{H^m(K)}^2
&\nonumber\le C \sum_{a_j\in \calV_K^B\cup \calM_K^B} h_{e_j}^{4-2m} |\mu(a_j)|^2
\le C \mathop{\sum_{e\in \calE_h^B}}_{a_j\in \bar e:\ a_j\in \calV_K^B} h_e^{3-2m} \|\mu\|_{L^2(e)}^2.
\end{align*}
Therefore, by an inverse inequality and shape-regularity of $\calT_h^{ct}$,
\begin{align*}
\|\bv\|_{1,h}^2 
&= \|\nab \bv\|_{L^2(\Omega_h)}^2 + \sum_{e\in \calE_h^B} \frac1{h_e} \|\bv\|_{L^2(e)}^2\\
&\le C \|\mu\|_{-1/2,h}^2+C \sum_{K\in \calT_h^{ct}} h_K^{-2} \|\bv\|_{L^2(K)}^2\le C \|\mu\|_{-1/2,h}^2.
\end{align*}
Combining this estimate with \eqref{ineq0}
yields the desired inf-sup condition \eqref{eqn:LinearInfSup}.
\end{proof}

\begin{remark}
The proof of Lemma \ref{lem:LMStab}, and in particular
the proof of the claim, relies on the continuity properties
of the Lagrange multiplier space at nearly flat corner vertices.
\end{remark}

\subsection{Main Stability Estimates}

Combining Lemmas \ref{lem:LBB} and \ref{lem:LMStab}
yields inf-sup stability for the bilinear form $b_h(\cdot,\cdot)$.
We also show that this result implies inf-sup 
stability for the bilinear form with boundary correction $b_h^e(\cdot,\cdot)$.
\begin{theorem}\label{thm:MainIS}
Then there exists $\beta>0$ depending  only on $\beta_1$ and $\beta_2$ such that
\begin{align}\label{eqn:MainIS}
\beta \|(q,\mu)\|
\le \sup_{\bv\in \bV_h\backslash \{0\}} \frac{b_h(\bv,(q,\mu))}{\|\bv\|_{1,h}}\qquad \forall (q,\mu)\in \mathring{Q}_h\times \mathring{X}_h.
\end{align}
\end{theorem}
\begin{proof}
We use  Lemmas \ref{lem:LBB} and \ref{lem:LMStab} and
follow the arguments in \cite[Theorem 3.1]{HowellWalkington11}.

Fix $(q,\mu)\in \mathring{Q}_h\times \mathring{X}_h$.
The statement \eqref{eqn:LinearInfSup}
implies the existence of $\bv_2\in \bV_h$
such that $\|\bv_2\|_{1,h} \leq 1$ and
\begin{align*}
\int_{\p\Omega_h} (\bv_2\cdot \bn)\mu\, ds
\ge  \beta_2\| \mu\|_{-1/2,h}.
\end{align*}
By Lemma \ref{lem:LBB}, there exists $\bv_1\in \mathring{\bV}_h$ satisfying $\|\nab \bv_1\|_{L^2(\Omega_h)} = \|\bv_1\|_{1,h} \leq 1$ and
\begin{align*}
-\int_{\Omega_h} (\Div \bv_1)q \ge \beta_1 \|q\|_{L^2(\Omega_h)}.
\end{align*}

Set  $\bv = c\bv_1 +\bv_2$ for some $c>0$, so that $\|\bv\|_{1,h}\le (1+c)$,
and
\begin{align*}
-\int_{\Omega_h} (\Div \bv)q\, dx
& \ge  c \beta_1\|q\|_{L^2(\Omega_h)} - \|\Div \bv_2\|_{L^2(\Omega_h)} \|q\|_{L^2(\Omega_h)}\\
& \ge   c \beta_1 \|q\|_{L^2(\Omega_h)} -\sqrt{2} \|\nab \bv_2\|_{L^2(\Omega_h)} \|q\|_{L^2(\Omega_h)}\\
& \ge c \beta_1\|q\|_{L^2(\Omega_h)} -\sqrt{2} \|\bv_2\|_{1,h} \|q\|_{L^2(\Omega_h)}\\
& =\big(c \beta_1-\sqrt{2}\big)\|q\|_{L^2(\Omega_h)}.
\end{align*}
Because $\bv_1|_{\p\Omega} = 0$, we have
\begin{align*}
\int_{\p\Omega_h} (\bv\cdot \bn_h)\mu\, ds 
=\int_{\p\Omega_h} (\bv_2\cdot \bn_h)\mu\, ds 
\ge \beta_2 \|\mu\|_{-1/2,h}. 
\end{align*}
Therefore,
\begin{align*}
b_h(\bv,(q,\mu)) 
&\ge \big(c \beta_1-\sqrt{2}\big)\|q\|_{L^2(\Omega_h)}+\beta_2\|\mu\|_{-1/2,h}\\
&\ge (1+c)^{-1}\Big( \big(c \beta_1-\sqrt{2}\big)\|q\|_{L^2(\Omega_h)}+\beta_2\|\mu\|_{-1/2,h}
\Big)\|\bv\|_{1,h}.
\end{align*}
We now choose $c>0$ sufficiently large to obtain the desired result.
\end{proof}

\begin{corollary}\label{cor:EIS}
Provided $c_\delta$ in Assumption \eqref{eqn:Assumption} is sufficiently small, there exists
$\beta_e>0$ independent of $h$ such that
there holds
\begin{align}\label{eqn:mainISPert}
\beta_e \|(q,\mu)\|
\le  \sup_{\bv\in \bV_h\backslash \{0\}} \frac{b_h^e(\bv,(q,\mu))}{\|\bv\|_{1,h}}\qquad \forall (q,\mu)\in \mathring{Q}_h\times \mathring{X}_h.
\end{align}
\end{corollary}
\begin{proof}
Combining Theorem \ref{thm:MainIS} and Lemma \ref{lem:bhPert}, we have
\begin{align*}
\beta \|(q,\mu)\|
\le  \sup_{\bv\in \bV_h\backslash \{0\}} \frac{b^e_h(\bv,(q,\mu))}{\|\bv\|_{1,h}}+ C c_\delta \|(q,\mu)\|\qquad \forall (q,\mu)\in \mathring{Q}_h\times \mathring{X}_h.
\end{align*}
This result implies \eqref{eqn:mainISPert} for $c_\delta$ sufficiently small
with $\beta_e = \beta-Cc_\delta$.
\end{proof}

\begin{theorem}
Let $(\bu_h,p_h,\lambda_h)\in \bV_h\times \mathring{Q}_h\times \mathring{X}_h$
satisfy \eqref{eqn:LMM}.
Then, provided $c_\delta$ in Assumption \eqref{eqn:Assumption}
is sufficiently small, there holds
\begin{equation}\label{eqn:MainStabLine}
\nu \|\bu_h\|_{1,h}+\|(p_h,\lambda_h)\| \le C \|{\bm f}\|_{-1,h},
\end{equation}
where $\|{\bm f}\|_{-1,h} = \sup_{\bv\in \bV_h\backslash \{0\}} \frac{\int_{\Omega_h} {\bm f}\cdot \bv\, dx}{\|\bv\|_{1,h}}$.
Consequently, there exists a unique solution to \eqref{eqn:LMM}.
\end{theorem}
\begin{proof}
Setting $\bv= \bu_h$ in \eqref{eqn:LMM1}, $(q,\mu) = (p_h,\lambda_h)$ in \eqref{eqn:LMM2},
and subtracting the resulting expressions yields
\begin{alignat*}{2}
a_h(\bu_h,\bu_h) & = \int_{\Omega_h} {\bm f}\cdot \bu_h\, dx +\int_{\p\Omega_h} \big((S_h \bu_h - \bu_h) \cdot \bn_h\big) \lambda_h\, ds.
\end{alignat*}

We apply the coercivity result in Lemma \ref{lem:coercivity},
the Cauchy-Schwarz inequality, and \eqref{eqn:ShvDiff} to get
\begin{align}\label{eqn:Energy}
\nu c_1 \|\bu_h\|_{1,h}^2
& \le \|{\bm f}\|_{-1,h} \|\bu_h\|_{1,h} + C c_\delta \|\bu_h\|_{1,h}\|\lambda_h\|_{-1/2,h}.
\end{align}

On the other hand, we use inf-sup stability \eqref{eqn:MainIS} to conclude
\begin{align*}
 \beta \|(p_h,\lambda_h)\|_{-1/2,h}
&\le \sup_{\bv\in \bV_h\backslash \{0\}} \frac{b_h(\bv,(p_h,\lambda_h))}{\|\bv\|_{1,h}}\\
&\le \sup_{\bv\in \bV_h\backslash \{0\}} \frac{\int_{\Omega_h} {\bm f}\cdot \bv \, dx -a_h(\bu_h,\bv)}{\|\bv\|_{1,h}}.
\end{align*}
Using the continuity estimate \eqref{eqn:continuity} gets
\begin{align}\label{eqn:LineFEX}
\beta \|\lambda_h\|_{-1/2,h}\le \beta \|(p_h,\lambda_h)\|\le  \|{\bm f}\|_{-1,h} + c_2(1+\sigma) \nu \|\bu_h\|_{1,h}.
\end{align}
Inserting this estimate into \eqref{eqn:Energy}, we obtain
\begin{align*}
\nu \big(c_1 - C c_\delta c_2\beta^{-1}(1+\sigma)\big) \|\bu_h\|_{1,h} \le (1+C c_\delta \beta^{-1})\|{\bm f}\|_{-1,h}.
\end{align*}
Thus, $\|\bu_h\|_{1,h}\le C \nu^{-1}  \|{\bm f}\|_{-1,h}$ for $c_\delta$ sufficiently small.
This, combined with \eqref{eqn:LineFEX}, yields the desired stability result \eqref{eqn:MainStabLine}.

\end{proof}

\section{Convergence Analysis}\label{sec-convergence}

In this section, we show that 
the solution to the the finite element 
method \eqref{eqn:LMM} converges with optimal order
provided the exact solution is sufficiently smooth.
As a first step, we derive some consistency
estimates for the boundary correction
operator and the bilinear form $a_h(\cdot,\cdot)$.

\subsection{Consistency Estimates}
The following lemma bounds
the boundary correction operator acting on the exact velocity function.
The result is essentially an estimate on the Taylor polynomial remainder
and follows directly from the arguments in \cite[Proposition 3]{ACS20C} (also see \cite{BDT72}).
For this reason, its proof is omitted.
\begin{lemma}\label{lem:Consistency1}
For any $\bu\in \bH^3(\Omega)\cap \bH^1_0(\Omega)$, there holds
\begin{align*}
\sum_{e\in \calE_h^B} h_e^{-1} \int_e \big|S_h \bu\big|^2\, ds\le C h^4 \|\bu\|_{H^3(\Omega)}^2.
\end{align*}
\end{lemma}

\begin{lemma}\label{lem:LapCons}
There holds for all $\bu\in \bH^3(\Omega)\cap \bH^1_0(\Omega)$,
\begin{align}\label{eqn:LapCons1}
\Big|-\nu\int_{\Omega_h} \Delta \bu\cdot \bv\, dx - a_h(\bu,\bv)\Big|
&\le C \nu h^2 \|\bu\|_{H^3(\Omega)} \|\bv\|_{1,h}\qquad \forall \bv\in \bV_h.
\end{align}
If $\Div \bu = 0$ in $\Omega$, then
\begin{align*}
\big|b_h^e(\bu,(q,\mu))\big|&\le C h^2 \|\bu\|_{H^3(\Omega)} \|(q,\mu)\|\qquad \ \forall (q,\mu)\in \mathring{Q}_h\times \mathring{X}_h.
\end{align*}
\end{lemma}
\begin{proof}
We integrate-by-parts to write
\[
\Big|-\nu\int_{\Omega_h} \Delta \bu\cdot \bv\, dx - a_h(\bu,\bv)\Big| =\nu \Big| \sum_{e\in \calE_h^B} \int_e \frac{\partial \bv}{\partial \bn_h}(S_{h}\bu)\, ds+\sum_{e\in \calE_h^B} \frac{\sigma}{h_e} \int_{e}(S_h \bu) (S_h \bv)\, ds  \Big|.
\]
Next, we estimate the two terms on the right hand side of the above equality by using 
the Cauchy-Schwarz inequality, trace and inverse inequalities, along with Lemmas \ref{lem:EquivNorms} and \ref{lem:Consistency1} as follows:
\begin{align*}
\Big| \sum_{e\in \calE_h^B} \int_{e} \frac{\partial \bv}{\partial \bn_h}(S_{h}\bu)\, ds\Big|
&\le \Big( \sum_{e\in \calE_h^B} h_e \int_{e}\big|\frac{\partial \bv}{\partial \bn_h}\big|^2\, ds  \Big)^{1/2} \Big(\sum_{e\in \calE_h^B} h^{-1}_e \int_{e}|S_h \bu|^2\, ds  \Big)^{1/2}\\
&\le Ch^2\|\bu\|_{H^3(\Omega)}\|\bv\|_{1,h},
\end{align*}
and 
\begin{align*}
    \Big|\sum_{e\in \calE_h^B}\frac{\sigma}{h_e}  \int_{e}(S_h \bu) (S_h \bv)\, ds\Big| & \le \sigma \Big( \sum_{e\in \calE_h^B} h^{-1}_e \int_{e}|S_h \bu|^2\, ds  \Big)^{1/2} \Big( \sum_{e\in \calE_h^B} h^{-1}_e \int_{e}|S_h \bv|^2\, ds  \Big)^{1/2} \\
& \le Ch^2 \|\bu\|_{H^3(\Omega)}\|\bv\|_{1,h}.
\end{align*}
Thus, the first estimate \eqref{eqn:LapCons1} holds.

Similarly, another use of the Cauchy-Schwarz inequality with Lemma \ref{lem:Consistency1} yields
\[
\big|b_h^e(\bu,(q,\mu))\big|=\Big|\sum_{e\in \calE_h^B} \int_{e}(S_h \bu \cdot \bn_h)\mu\, ds \Big| \le Ch^2\|\bu\|_{H^3(\Omega)}\|\mu\|_{-1/2,h},
\]
and this completes the proof.
\end{proof}

\subsection{Approximation Properties of the Kernel}
We define the discrete kernel as
\[
\bZ_h = \{\bv\in \bV_h:\ b_h^e(\bv,(q,\mu))=0,\ \forall (q,\mu)\in \mathring{Q}_h\times \mathring{X}_h\}.
\]
Note that if $\bv\in \bZ_h$, then $\Div \bv=0$ in $\Omega_h$ (cf.~Lemma \ref{lem:DivFree}),
and 
\begin{equation}\label{eqn:KProp}
\int_{\p\Omega_h} ((S_h \bv)\cdot \bn_h)\mu\, ds=0\qquad \forall \mu\in \mathring{X}_h.
\end{equation}

In this section, we show
that the kernel $\bZ_h$ has optimal order approximation 
properties with respect to divergence-free smooth functions.
To this end, we define the orthogonal complement of $\bZ_h$ as
\[
\bZ_h^\perp:= \{\bv\in \bV_h:\ (\bv,\bw)_{1,h} = 0\ \ \forall \bw\in \bZ_h\},
\]
where $(\cdot,\cdot)_{1,h}$ is the inner product on $\bV_h$
that induces the norm $\|\cdot\|_{1,h}$. 

\begin{lemma}\label{lem:AIS}
There holds
\begin{align*}
{\beta}_e \|\bw\|_{1,h}\le \sup_{(q,\mu)\in \mathring{Q}_h\times \mathring{X}_h\backslash \{0\}} 
\frac{b_h^e(\bw,(q,\mu))}{\|(q,\mu)\|}\qquad \forall \bw\in \bZ_h^\perp.
\end{align*}
\end{lemma}
\begin{proof}
The result follows from Corollary \ref{cor:EIS}
and standard results in mixed finite element theory (cf.~\cite[Lemma 12.5.10]{BrennerScott08}).
%
\end{proof}

The following theorem states
the approximation properties of the discrete kernel.

\begin{theorem}\label{thm:ApproxKer}
For any $\bu\in \bH^3(\Omega)\cap \bH^1_0(\Omega)$
with $\Div \bu = 0$, there holds
\begin{align}\label{eqn:ApproxK}
\inf_{\bw\in \bZ_h} \tbar{\bu-\bw}_{h}\le C h^2 \|\bu\|_{H^3(\Omega)}.
\end{align}
\end{theorem}
\begin{proof}
Let $\bv\in \bV_h$ be arbitrary.
By Corollary \ref{cor:EIS}, there exists
$\by\in \bZ^\perp_h$ such that
\[
b_h^e(\by,(q,\mu)) = b_h^e(\bu-\bv,(q,\mu))\qquad \forall (q,\mu)\in \mathring{Q}_h\times \mathring{X}_h,
\]
and $\|\by\|_{1,h}\le C {\beta}_e^{-1}\|\bu-\bv\|_{1,h}$.
We then let $\bz\in \bZ_h^\perp$ satisfy
\[
b_h^e(\bz,(q,\mu)) = -b_h^e(\bu,(q,\mu))\qquad \forall (q,\mu)\in \mathring{Q}_h\times \mathring{X}_h.
\]
Then $\bw:=\bv+\by+\bz\in \bZ_h$, and 
\begin{align*}
\|\bu-\bw\|_{1,h}
&\le \|\bu-\bv\|_{1,h}+\|\by\|_{1,h}+\|\bz\|_{1,h}\\
&\le (1+C{\beta}_e^{-1})\|\bu-\bv\|_{1,h} +\|\bz\|_{1,h}.
\end{align*}

By Lemmas \ref{lem:AIS} and \ref{lem:LapCons},
\begin{align*}
\beta_e \|\bz\|_{1,h}
&\le \sup_{(q,\mu)\in \mathring{Q}_h\times \mathring{X}_h\backslash \{0\} }
\frac{b_h^e(\bu,(q,\mu))}{\|(q,\mu)\|}\le C h^2 \|\bu\|_{H^3(\Omega)},
\end{align*}
and so, by Lemma \ref{lem:EquivNorms},
\begin{align*}
\tbar{\bu-\bw}_{h}
&\le
\tbar{\bu-\bv}_{h}
+C\|\bv-\bw\|_{1,h}
\le C\big(\tbar{\bu-\bv}_{h}+\|\bu-\bw\|_{1,h}\big)\\
&\le C(1+\beta_e^{-1})\big(\tbar{\bu-\bv}_h+ \|\bu-\bv\|_{1,h} + h^2 \|\bu\|_{H^3(\Omega)}\big)\qquad \forall \bv\in \bV_h.
\end{align*}
Taking $\bv$ to be the nodal interpolant of $\bu$, we obtain the desired result. 
\end{proof}

\begin{theorem}\label{thm:MainConvResult}
Suppose that the solution to \eqref{eqn:Stokes}
has regularity $(\bu,p)\in \bH^3(\Omega)\cap \bH^1_0(\Omega)\times H^s(\Omega)$
 for some $1\le s \le 3$. Furthermore, without loss of generality, assume that $p|_{\Omega_h}\in L^2_0(\Omega_h)$.
 Then,
 \begin{subequations}
 \label{eqn:MainEst}
\begin{align}
 \label{eqn:MainEst1}
\|\bu-\bu_h\|_{1,h}&\le C\big(h^2 \|\bu\|_{H^3(\Omega)}+ \nu^{-1} \inf_{\mu\in X_h} \|p-\mu\|_{-1/2,h}\big),\\
 \label{eqn:MainEst2}
\|p-p_h\|_{L^2(\Omega_h)}&\le C(\nu h^2 \|\bu\|_{H^3(\Omega)}+\inf_{\mu \in X_h}\|p-\mu\|_{-1/2,h}+\inf_{q_h \in \mathring{Q}_h}\|p-q_h\|_{L^2(\Omega)}),\\
 \label{eqn:MainEst3}
\|\lambda_h-\mathring{\mu}\|_{-1/2,h}&\le C \big(\nu h^2 \|\bu\|_{H^3(\Omega)}+\|p-\mu\|_{-1/2,h}\big)\qquad \forall \mu\in X_h,
\end{align}
\end{subequations}
where $\mathring{\mu}:= \mu - \frac1{|\p\Omega_h|} \int_{\p\Omega_h} \mu\, ds$.
In particular there holds
\begin{align*}
\|\bu-\bu_h\|_{1,h} &\le C\big(h^2 \|\bu\|_{H^3(\Omega)}+ \nu^{-1} h^s \|p\|_{H^s(\Omega)}\big),\\
\|p-p_h\|_{L^2(\Omega_h)} &\le C\big(\nu h^2 \|\bu\|_{H^3(\Omega)}+h^{\min\{2,s\}}\|p\|_{H^{\min\{2,s\}}(\Omega)}\big).
\end{align*}
\end{theorem}
\begin{proof}
Let $\bw\in \bZ_h$ be arbitrary.
We then have, for all $\bv\in \bZ_h$ and $\mu\in {X}_h$, 
\begin{align*}
a_h(\bu_h-\bw,\bv)
& = \int_{\Omega_h} {\bm f}\cdot \bv - a_h(\bw,\bv) - b_h(\bv,(p_h,\lambda_h))\\
& = -\nu \int_{\Omega_h}  \Delta \bu \cdot \bv\, dx - a_h(\bw,\bv) -\int_{\p\Omega_h} (\bv\cdot \bn_h)(p-\lambda_h)\, ds\\
& =-\nu \int_{\Omega_h}  \Delta \bu \cdot \bv\, dx - a_h(\bw,\bv) -\int_{\p\Omega_h} (\bv\cdot \bn_h)(p-\mu)\, ds
+ \int_{\p\Omega_h} (\bv\cdot \bn_h)(\lambda_h-\mathring{\mu})\, ds,
\end{align*}
where $\mathring{\mu} = \mu - \frac1{|\p\Omega_h|} \int_{\p\Omega_h} \mu\, ds\in \mathring{X}_h$.

Therefore by Lemma \ref{lem:LapCons},
the continuity of $a_h(\cdot,\cdot)$ (cf.~\eqref{eqn:continuity}),
and the Cauchy-Schwarz inequality, 
\begin{align*}
a_h(\bu_h-\bw,\bv)
&\le  C \big(\nu h^2 \|\bu\|_{H^3(\Omega)} + \|p-\mu\|_{-1/2,h}\big) \|\bv\|_{1,h} +
a_h(\bu-\bw,\bv) + \int_{\p\Omega_h} (\bv\cdot \bn_h)(\lambda_h-\mathring{\mu})\, ds\\
&\le  C \big(\nu h^2 \|\bu\|_{H^3(\Omega)} + \nu(1+\sigma) \tbar{\bu-\bw}_{h}+ \|p-\mu\|_{-1/2,h}\big) \|\bv\|_{1,h} 
+ \int_{\p\Omega_h} (\bv\cdot \bn_h)(\lambda_h-\mathring{\mu})\, ds.
\end{align*}

We then use \eqref{eqn:KProp} and 
\eqref{eqn:ShvDiff} to obtain
\begin{align*}
\int_{\p\Omega_h} (\bv\cdot \bn_h)(\lambda_h-\mathring{\mu})\, ds = 
\int_{\p\Omega_h} \big((\bv-S_h \bv)\cdot \bn_h\big)(\lambda_h-\mathring{\mu})\, ds\le C c_\delta \|\bv\|_{1,h} 
\|\lambda_h-\mathring{\mu}\|_{-1/2,h}.
\end{align*}

Setting $\bv = \bu_h-\bw$, applying the coercivity of $a_h(\cdot,\cdot)$ and Theorem \ref{thm:ApproxKer}, we obtain
\begin{equation}\label{eqn:MainEstPre}
c_1 \nu \|\bu_h-\bw\|_{1,h}
\le  C \big(\nu (1+\sigma) h^2 \|\bu\|_{H^3(\Omega)} + \|p-\mu\|_{-1/2,h}
+c_\delta  \|\lambda_h-\mathring{\mu}\|_{-1/2,h}\big)
\end{equation}
for $\bw\in \bZ_h$ satisfying \eqref{eqn:ApproxK}.

Next, let $P_h\in \mathring{Q}_h$ be the $L^2$-projection of $p$ and note that,
due to the definitions of the finite element spaces, $\int_{\Omega_h} (\Div \bv)(p-P_h)\, dx=0$
for all $\bv\in \bV_h$.
This identity, along with the inf-sup stability estimate given in Theorem \ref{thm:MainIS} yields 
\begin{align*}
\beta \|(p_h-P_h,\lambda_h-\mathring{\mu})\|
&\le \sup_{\bv\in \bV_h\backslash \{0\}} \frac{b_h(\bv,(p_h-P_h,\lambda_h-\mathring{\mu}))}{\|\bv\|_{1,h}}
= \sup_{\bv\in \bV_h\backslash \{0\}} \frac{b_h(\bv,(p_h-p,\lambda_h-\mathring{\mu}))}{\|\bv\|_{1,h}}.
\end{align*}
Using Lemma \ref{lem:LapCons}, we write the numerator as
\begin{align*}
b_h(\bv,(p_h-p,\lambda_h-\mathring{\mu}))
&= b_h(\bv,(p_h,\lambda_h)) - b_h(\bv,(p,\mathring{\mu}))\\
& = \int_{\Omega_h} {\bm f}\cdot \bv\, dx - a_h(\bu_h,\bv) +\int_{\Omega_h} (\Div \bv)p\, dx
-\int_{\p\Omega_h} (\bv\cdot \bn_h){\mu}\, ds\\
%
%
& \le C \nu h^2 \|\bu\|_{H^3(\Omega)}+ a_h(\bu-\bu_h,\bv) 
-\int_{\p\Omega_h} (\bv\cdot \bn_h)(\mu-p)\, ds.
\end{align*}
By continuity and the Cauchy-Schwarz inequality,
\begin{align}\label{eqn:phProjection}
&\beta \|(p_h-P_h,\lambda_h-\mathring{\mu})\|
\le C \big(\nu h^2 \|\bu\|_{H^3(\Omega)}+ c_2\nu (1+\sigma) \tbar{\bu-\bu_h}_{h} + \|p-\mu\|_{-1/2,h}\big)\\
&\quad\nonumber\le C \big(\nu h^2 \|\bu\|_{H^3(\Omega)}+ c_2 \nu (1+\sigma) \big(\tbar{\bu-\bw}_{h} + \|\bu_h-\bw\|_{1,h}\big)+ \|p-\mu\|_{-1/2,h}\big)\\
&\quad\nonumber\le C \big(\nu(1+\sigma) h^2 \|\bu\|_{H^3(\Omega)}+ c_2 \nu (1+\sigma) \|\bu_h-\bw\|_{1,h}+ \|p-\mu\|_{-1/2,h}\big).
\end{align}
Inserting this estimate into \eqref{eqn:MainEstPre}, we get
\begin{align}\label{eqn:FDSA}
\nu \big(c_1 - C \beta ^{-1} c_2 (1+\sigma) c_\delta\big)\|\bu_h-\bw\|_{1,h}
\le C \nu (1+\sigma)h^2 \|\bu\|_{H^3(\Omega)} + \|p-\mu\|_{-1/2,h}.
\end{align}
Using he approximation properties of the discrete kernel once again(cf.~Theorem \ref{thm:ApproxKer}), and for $c_\delta$ sufficiently small,
\begin{align*}
\|\bu-\bu_h\|_{1,h}\le C \big(h^2 \|\bu\|_{H^3(\Omega)} +\nu^{-1} \inf_{\mu\in X_h} \|p-\mu\|_{-1/2,h}\big).
\end{align*}
This establishes the velocity estimate \eqref{eqn:MainEst1}.

To obtain the estimate for the pressure approximation \eqref{eqn:MainEst2}, we
use the triangle inequality and the approximation properties of the $L^2$-projection:
\[
\|p-p_h\|_{L^2(\Omega_h)} \le \|p_h-P_h\|_{L^2(\Omega_h)}+\inf_{q_h \in \mathring{Q}_h}\|p-q_h\|_{L^2(\Omega_h)}.
\]
Inserting \eqref{eqn:phProjection} and \eqref{eqn:FDSA} into the right-hand side 
yields the desired
bound for the pressure.  Likewise, combining \eqref{eqn:phProjection}
and \eqref{eqn:FDSA} yields \eqref{eqn:MainEst3}.
\end{proof}

\section{Numerical Experiments}\label{sec-Numerics}

\begin{figure}
\centering
\includegraphics[width=2.5in,height=2.5in]{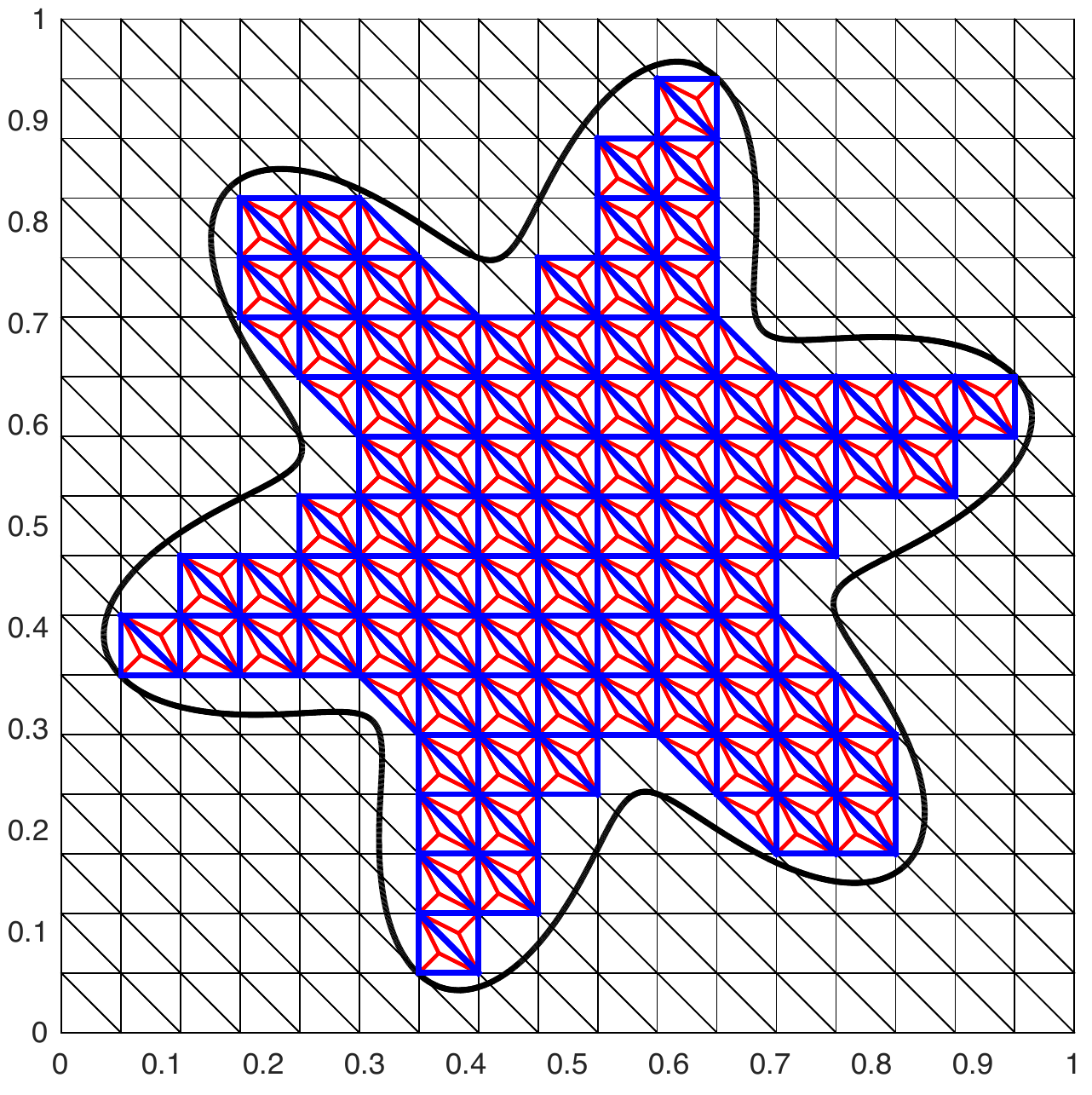}
\includegraphics[width=2.5in,height=2.5in]{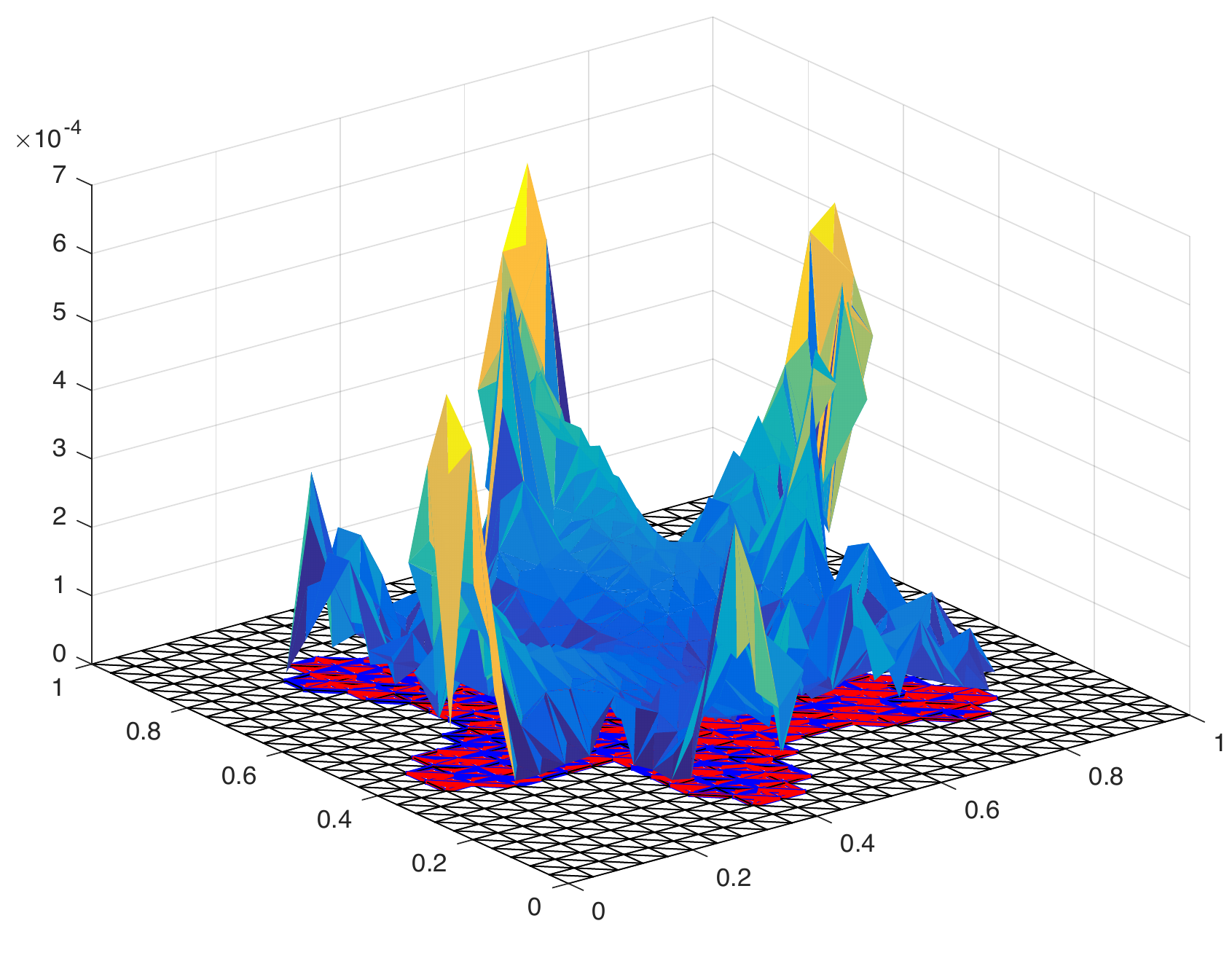}
\caption{\label{fig:Mesh}Left: The domain and mesh with $h=1/24$. 
Right: The 
graph of the error $|\bu-\bu_h|$ with exact solution \eqref{eqn:Test1ExactSoln}.}
\end{figure}

In this section we perform simple numerical experiments
of the finite element method
\eqref{eqn:LMM} which verify the theoretical rates of convergence established
in the previous sections.

In the series of tests, the domain is defined via a level set function \cite{Lehrenfeld16}
\begin{align}\label{eqn:Test1Domain}
\Omega = \{x\in \bbR^2:\ \phi(x)<0\}\text{, where}\quad
\phi = r-0.3723423423343-0.1\sin(6\theta),
\end{align}
with $r = \sqrt{(x_1-0.5)^2+(x_2-0.5)^2}$,
and $\theta = \tan^{-1}((x_2-0.5)/(x_1-0.5))$.
We take $S=(0,1)^2$, and the background
mesh $\calS_h$ to be a sequence of type I triangulations
of $S$, i.e., a mesh obtained by drawing
diagonals of a cartesian mesh; cf.~Figure \ref{fig:Mesh}.
For all tests, the Nitsche penalty parameter 
in the bilinear form $a_h(\cdot,\cdot)$ takes the value $\sigma = 40$.

The extension direction ${\bm d}$ is obtained
by solving an auxiliary $2\times 2$ nonlinear system at each
quadrature point of each boundary edge of $\calT_h^{ct}$.
In particular, for each quadrature point $x\in \p\Omega_h$, we
find $x_*\in \p\Omega$ such that
\[
\phi(x_*) = 0,\quad (\nab \phi(x_*))^\perp \cdot (x-x_*) = 0,
\]
and set ${\bm d} = (x-x_*)/|x-x_*|$ and $\delta(x) = |x-x_*|$.
The first equation ensures that $x_*$ is on the boundary $\p\Omega$,
whereas the second equation states that ${\bm d}$ is parallel
to the outward unit normal of $\p\Omega$ at $x$.

We choose the data such that
the exact solution to the Stokes problem is given by
\begin{align}\label{eqn:Test1ExactSoln}
\bu = \begin{pmatrix}
2(x_1^2-x_1+\frac14+x_2^2-x_2)(2 x_2-1)\\
-2 (x_1^2-x_1+\frac14+x_2^2-x_2)(2 x_1-1)
\end{pmatrix},\quad
p = 10(x_1^2-x_2^2)^2.
\end{align}
Because the exact solution is smooth,
Theorem \ref{thm:MainConvResult}
predicts the convergence rates
\begin{align}\label{eqn:RatesEx}
\|\nab (\bu-\bu_h)\|_{L^2(\Omega_h)} = \mathcal{O}(h^2+ \nu^{-1} h^3),\qquad
\|p-p_h\|_{L^2(\Omega_h)} = \mathcal{O}(h^2).
\end{align}

The velocity and pressure errors
are plotted in Figure \ref{fig:ErrorPlots}
for mesh parameters $h = 2^{-j}$\ ($j=3,4,5,6,7$)
and viscosities $\nu =10^{-k}$\ ($k=1,3,5$).
The results show that, for the moderately sized viscosities $\nu=10^{-1}$ and $\nu = 10^{-3}$, 
the $L^2$ and $H^1$ velocities converge with the optimal order
three and two, respectively.   We also observe larger
velocity errors for  viscosity value $\nu = 10^{-5}$, although,
rates of convergence are higher; Figure \ref{fig:ErrorPlots}
shows fourth and third order convergence in the $L^2$ and $H^1$ norms.
This behavior is consistent with the theoretical estimate \eqref{eqn:RatesEx}.
Finally, the numerical experiments show second order convergence
for the pressure approximation (with only marginal differences for different
viscosity values) and divergence errors comparable to machine epsilon.

\begin{center}
\begin{figure}
\begin{tikzpicture}[scale=0.9]
\begin{loglogaxis}[title = {$\|\bu-\bu_h\|_{L^2(\Omega_h)}$},xlabel={$h$},ylabel={},legend pos=south west,x dir=reverse]


\addplot coordinates {
(0.125,  4.89743e-03)
(0.062,  1.69819e-04)  
(0.029,  2.07378e-05)  
(0.016,  2.67332e-06)  
(0.008,  4.21456e-07)  
};

\addplot coordinates {
(0.125,  1.13930e-02)
(0.062,  4.58909e-04)
(0.029,  3.57077e-05)  
(0.016,  4.08034e-06)
(0.008,  4.70075e-07)
};

\addplot coordinates {
(0.125,  1.17523e+00)
(0.062,  4.26709e-02) 
(0.029,  2.95412e-03) 
(0.016,  3.08615e-04) 
(0.008,  2.07830e-05) 
};


\legend{$\nu=10^{-1}$,$\nu=10^{-3}$,$\nu=10^{-5}$}
\end{loglogaxis}
\end{tikzpicture}\quad
\begin{tikzpicture}[scale=0.9]
\begin{loglogaxis}[title={$\|\nab (\bu-\bu_h)\|_{L^2(\Omega_h)}$},xlabel={$h$},ylabel={},legend pos=north east,x dir=reverse]


\addplot coordinates {
(0.125,  9.43747e-02)
(0.062,  8.68162e-03)
(0.029,  2.01864e-03)
(0.016,  5.51167e-04)
(0.008,  1.39274e-04)
};

\addplot coordinates {
(0.125,  1.99512e-01)
(0.062,  1.35781e-02)  
(0.029,  2.32173e-03) 
(0.016,  5.63735e-04)  
(0.008,  1.39636e-04) 
};

\addplot coordinates {
(0.125,  2.07600e+01)
(0.062,  1.04710e+00) 
(0.029,  1.16325e-01)  
(0.016,  1.17446e-02)  
(0.008,  1.04713e-03) 
};


\legend{$\nu=10^{-1}$,$\nu=10^{-3}$,$\nu=10^{-5}$}
\end{loglogaxis}
\end{tikzpicture}\bigskip\\
\begin{tikzpicture}[scale=0.9]
\begin{loglogaxis}[title={$\|p-p_h\|_{L^2(\Omega_h)}$},xlabel={$h$},ylabel={},legend pos=north east,x dir=reverse]
%
\addplot coordinates {
(0.125,  1.75052e-01)
(0.062,  5.08658e-03) 
(0.029,  1.15511e-03) 
(0.016,  2.90247e-04) 
(0.008,  7.34069e-05) 
};

\addplot coordinates {
(0.125,  9.38550e-03)
(0.062,  2.75034e-03)  
(0.029,  7.17926e-04)  
(0.016,  2.18918e-04)
(0.008,  5.59244e-05)
};

\addplot coordinates {
(0.125,  9.44930e-03)
(0.062,  2.75002e-03)
(0.029,  7.17871e-04)
(0.016,  2.18910e-04)
(0.008,  5.59224e-05)
};

\legend{$\nu=10^{-1}$,$\nu=10^{-3}$,$\nu=10^{-5}$}
\end{loglogaxis}
\end{tikzpicture}\quad
\begin{tikzpicture}[scale=0.9]
\begin{loglogaxis}[title={$\|\Div \bu_h\|_{L^\infty(\Omega_h)}$},xlabel={$h$},ylabel={},legend pos=north east,x dir=reverse]
\addplot coordinates {
(0.125,  8.57092e-13)
(0.062,  2.39364e-13)
(0.029,  4.47642e-13)
(0.016,  7.43405e-13)
(0.008,  1.08713e-11)
};

\addplot coordinates {
(0.125,  9.16032e-11)
(0.062,  6.42508e-12) 
(0.029,  4.05365e-12) 
(0.016,  1.90781e-12) 
(0.008,  1.74296e-11)
};

\addplot coordinates {
(0.125,  9.16363e-09)
(0.062,  6.27226e-10)
(0.029,  4.27022e-10)
(0.016,  4.39302e-11)
(0.008,  5.71809e-12)
};

\legend{$\nu=10^{-1}$,$\nu=10^{-3}$,$\nu=10^{-5}$}
\end{loglogaxis}
\end{tikzpicture}
\caption{\label{fig:ErrorPlots}Errors for the velocity and pressure
for a sequence of meshes on domain \eqref{eqn:Test1Domain}
and exact solution \eqref{eqn:Test1ExactSoln}.}
\end{figure}
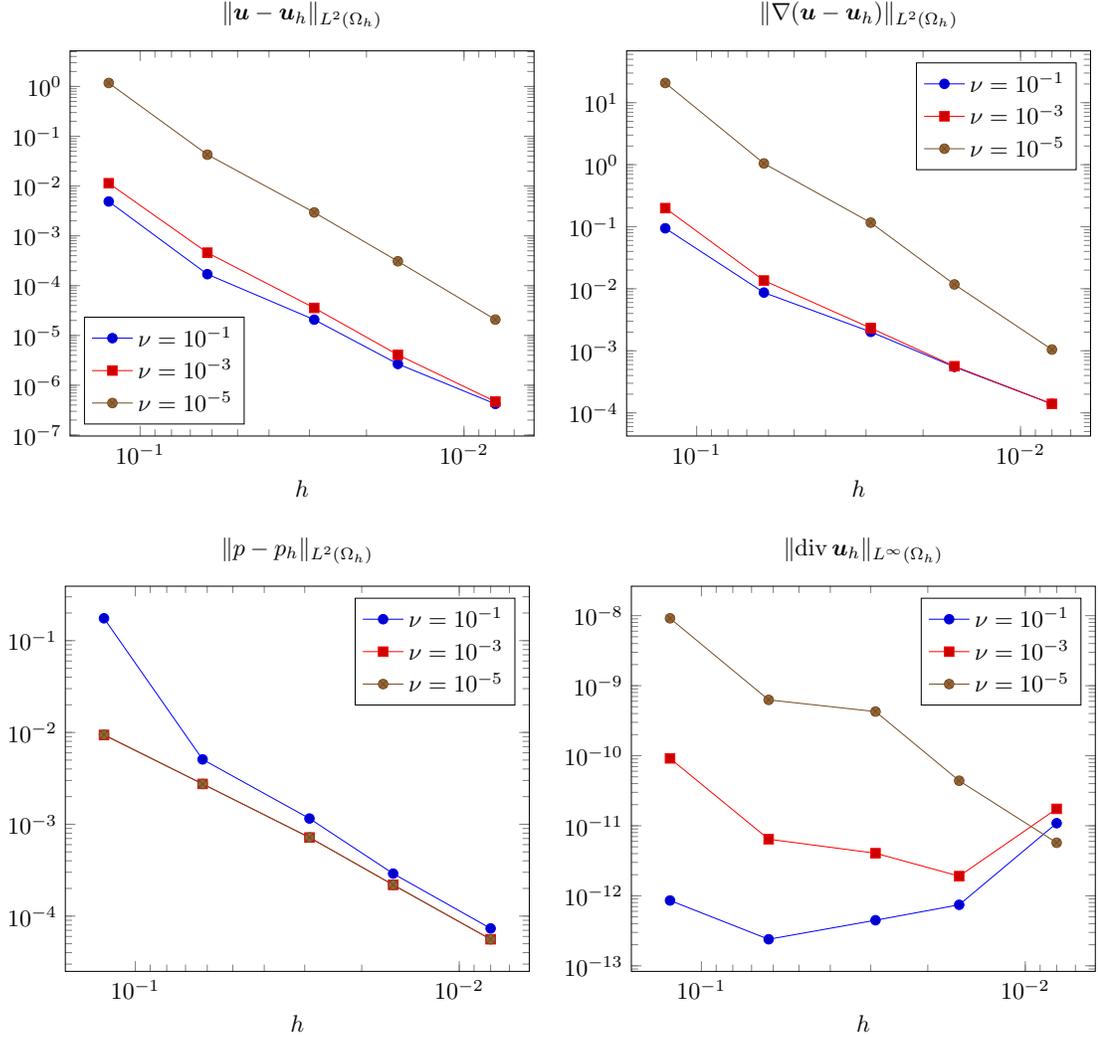
\end{center}


%
%
%
%
%
%
%

%
%
%
%
%
%

\section{Concluding Remarks}\label{sec-conclusion}
This paper constructed
a uniformly stable and divergence-free method 
for the Stokes problem on unfitted meshes 
using a boundary correction approach.  While the method
is not pressure-robust, a Lagrange multiplier
enforcing the normal boundary conditions
is included to mitigate the affect of the pressure
contribution in the velocity error.  Theoretical results and numerical experiments
show that the method converges with optimal order.

The presentation is confined
to the two dimensional setting, however
many of the results extend to 3D as well.
For example, the proof of inf-sup stability given in Lemma 
\ref{lem:LBB}
applies mutatis mutandis to the
the three-dimensional Scott-Vogelius  pair.
On the other hand, inf-sup stability 
of the velocity-Lagrange multiplier pairing (cf.~Lemma \ref{lem:LMStab}),
and its dependence on the geometry of the computational mesh
is less obvious.  We plan to address this issue 
in the near future.

\end{document}